\documentclass[11pt]{amsart}
\usepackage{amsmath, amssymb}
\usepackage{upgreek,tensor}
\usepackage{xcolor, url}
\usepackage[all]{xy}
\usepackage{graphicx}
\usepackage{subcaption}

\usepackage{longtable}
\usepackage{graphicx}
\usepackage{float} 
\usepackage{bm}
\usepackage{enumerate}
\theoremstyle{plain}
\newtheorem{theorem}{Theorem}[section]

\newtheorem{prop}[theorem]{Proposition}

\newcommand{\td}{\text{d}}
\theoremstyle{definition}
\newtheorem{remark}[theorem]{Remark}

\numberwithin{equation}{section}

\usepackage{graphics}

\usepackage{hyperref} 
\hypersetup{colorlinks=true, linkcolor=blue, citecolor = blue, urlcolor=blue, filecolor=magenta}

\begin{document}
\title[Stability of Homogeneous minimal hypersurfaces...]{Stability of Homogeneous minimal hypersurfaces in the Page space and $Y^{p,q}$ Sasaki-Einstein manifolds}
\author{Natalia Gherghel}
\address{Department of Physics and Astronomy\\
		McMaster University\\
		Hamilton, ON Canada}
	\email{gherghen@mcmaster.ca}
	\author{Hari K. Kunduri}
\address{Department of Mathematics and Statistics and Department of Physics and Astronomy\\
		McMaster University\\
		Hamilton, ON Canada}
	\email{kundurih@mcmaster.ca}


\thanks{N. Gherghel is partially supported by a Physics and Astronomy Graduate Scholarship from McMaster University. H. K. Kunduri is supported by an  NSERC Discovery Grant RGPIN-2025-06027.}

\begin{abstract}
We investigate the stability of homogeneous minimal submanifolds of co-dimension one (hypersurfaces) in two families of closed Einstein manifolds, the Page space $\mathbb{CP}^2 \# \overline{\mathbb{CP}^2}$ and the Sasaki-Einstein spaces $Y^{p,q}$, which are equipped with cohomogeneity-one Einstein metrics admitting the isometric action of $SU(2) \times U(1)$ and $U(1) \times U(1) \times SU(2)$ respectively. We determine all the homogeneous, minimal hypersurfaces and explicitly compute the spectrum of their associated stability operators and determine their index. 
\end{abstract}
\maketitle
\section{Introduction}
An Einstein manifold is a Riemannian manifold $(M, g)$ whose Ricci tensor satisfies $\text{Ric}(g) = c g$ for some constant $c \in \mathbb{R}$ (for a classic survey, see \cite{Besse}). For homogeneous manifolds, the Einstein condition reduces to an algebraic system and there is a well-developed existence theory. For example, each compact simply
connected homogeneous space of dimension less than 12 admits a homogeneous Einstein metric \cite{Bohm2, Bohm}. A successful approach to constructing inhomogeneous Einstein metrics is to consider \emph{cohomogeneity-one} Einstein metrics. A Riemannian manifold $(M,g)$ is said to be cohomogeneity-one if a Lie group $G$ acts isometrically on $M$ with principal orbits $G /K$ of codimension one (that is, the generic orbits of the isometry group are hypersurfaces). In this case, the Einstein condition reduces to a system of ordinary differential equations in the variable on the space of orbits. 

Einstein manifolds of positive curvature $c>0$ are necessarily compact with finite fundamental group by Myers' theorem \cite{Myers}. Explicit families of examples of such metrics on closed manifolds $M$ can be readily constructed. The first of these was Page's inhomogeneous (conformally K\"ahler) Einstein metric on $\mathbb{CP}^2 \# \overline{\mathbb{CP}}^2$, which was originally found by an analytic continuation of a Lorentzian Einstein metric describing a black hole spacetime \cite{Page}. The metric has a four-dimensional isometry group $SU(2) \times U(1)$ acting on the $S^3$ surfaces of homogeneity $\Sigma_t$ where $t\in (-1,1)$ parameterizes the level sets. The homogeneous geometry $(\Sigma_t,h_t)$ induced on these surfaces is not the standard round metric on $S^3$, but it is nonetheless possible to explicitly determine the spectrum of the Laplacian \cite{Wu:1976ge, Hennigar}. This can in turn be used to reduce the problem of determining the problem of computing the spectrum of the Laplacian on the Page metric to a one-dimensional Sturm-Liouville problem \cite{Hennigar}. 

A second class of explicit cohomogeneity-one positive Einstein metrics on closed manifolds are the countably infinite family of toric Sasaki-Einstein metrics $Y^{p,q}$ where $(p,q)$ are positive coprime  \cite{Gauntlett:2004yd} (recall that an odd-dimensional Riemannian manifold $(M,g)$ is said to admit a Sasakian structure if and only if its metric cone $C(S):= (\mathbb{R}_{0^+} \times M, \bar{g} = \td r^2 + r^2 g)$ is K\"ahler). The $Y^{p,q}$ metrics are defined on the product $S^2 \times S^3$ where now $G = SU(2) \times U(1) \times U(1)$ acting with four-dimensional orbits on the surfaces of homogeneity $S^1 \times L(2,1)$ (see Section 3 for details). These geometries represent the first examples of quasi-regular and irregular Sasaki-Einstein structures on $S^2 \times S^3$ (these properties refer to whether the orbits of a certain canonically defined Killing vector field are closed and whether the action it generates is free - see the comprehensive survey \cite{Sparks} for precise definitions). In addition to being of intrinsic interest in differential geometry, the $Y^{p,q}$ geometries arise naturally within theoretical physics. In particular, the spacetime AdS$_5 \times Y^{p,q}$ furnishes a supersymmetric solution of Type IIB supergravity. String theory on this spacetime background is expected to be equivalent to a certain family of four-dimensional superconformal field theories \cite{Martelli:2004wu}. Thus, the manifolds $Y^{p,q}$ provide an interesting class of examples to study the gauge theory-gravity (AdS/CFT) correspondence beyond the classic maximally symmetric setting of AdS$_5 \times S^5$ \cite{AdSCFT} (indeed, $S^5$ with its round metric is the model Sasaki-Einstein space). As in the Page space case, the spectrum of the Laplacian on the induced metric on the principal orbits of the isometry action can be determined explicitly, which allows one to reduce the problem of determining the spectrum of the Laplacian on $Y^{p,q}$ to analyzing an ordinary differential equation of Heun type~\cite{Kihara:2005nt}.

The purpose of the present note is to investigate the stability of a class of embedded minimal hypersurfaces within the two families of the Einstein manifolds discussed above. The cohomogeneity-one property plays an essential role in the analysis.  Recall that under a one parameter family of deformations $\Phi_t$ generated by the vector field $X = f n$ where $n$ is the outward unit normal and $f \in C^\infty (\Sigma)$, we can consider the family of deformations of a fixed two-sided (that is, there is a continuous choice of unit normal) embedded closed hypersurface $(\Sigma,h)$ by $\Sigma_t:=\Phi_t(\Sigma)$.  As is well known, the second variation of the area functional over a minimal hypersurface ($H_\Sigma =0$) is given by
\begin{align}\label{secondvar}
\frac{\td^2}{\td t^2}{\Bigg|}_{t=0} \text{Area}[\Sigma_t] & =\int_\Sigma \left(|\nabla f|^2 - (\text{Ric}(n,n) + \text{Tr} (K^2))f^2 \right)\; \td \text{vol}
\end{align} A compact minimal hypersurface is called stable if the above second variation is nonnegative for all $f \in C^\infty(\Sigma)$. The stability (Jacobi) operator $L$ associated with a minimal hypersurface is defined by
\begin{equation}
    L:= -\Delta_\Sigma - (\text{Ric}(n,n) + \text{Tr} (K^2))  = -\Delta_\Sigma + \frac{1}{2} (R(h) - R(g) - \text{Tr} (K^2))
\end{equation} where in the second equality we have used the trace of the Gauss equation with $H_\Sigma =0$. Here $R(h), R(g)$ denote the scalar curvatures of the induced metric $h$ and the ambient space metric $g$ respectively, and $\text{Tr} (K^2) = K_{ij} K^{ij}$. It is immediate by setting $f = 1$ in \eqref{secondvar} that if $\text{Ric}(g) >0$ then there are no closed stable two-sided minimal hypersurfaces. This applies to the Page manifold and the $Y^{p,q}$ spaces, which are positive Einstein.  

The stability operator $L$ can be shown to always have an infinite number of positive eigenvalues, but may or may not have a finite number of negative eigenvalues. The \emph{Morse index} of a minimal hypersurface, denoted $\text{index}(\Sigma)$, is defined to be the number of negative eigenvalues (assuming it is finite) of $L$. Equivalently \cite{Chodosh},
\begin{equation}
    \text{index}(\Sigma):=\sup \{ \text{dim} V | V \subset C^\infty(\Sigma), L < 0 \; \text{on } V \setminus \{0\} \}.
\end{equation} The index measures the sum of the dimensions of the eigenspaces corresponding
to negative eigenvalues of $L$, that is, the maximal dimension of the space of variations that destabilize $\Sigma$ to second order. 

In a space with positive Ricci curvature, the closed minimal hypersurfaces with index one are therefore the `most' stable. The index has proved quite important in the classification of minimal surfaces; for example, for immersed minimal surfaces in $\mathbb{R}^3$ with finite total curvature, only the plane has index 0 (it is stable), while the catenoid and Enneper's surface are the possibilities of index 1.  For a survey of results relating the geometry and topology of minimal surfaces, see \cite{survey}. 

A classic result of Simons establishes that for minimal spheres $S^p$ with the standard totally geodesic embedding in $S^n$, the index is $n-p$ and its nullity is $(p+1)(n-p)$ \cite[Prop. 5.1.1]{Simons}. More generally, for any compact, closed $p-$dimensional minimal immersed submanifold $\Sigma$ of $S^n$, $\text{index}(\Sigma) \geq n-p$ and its nullity $\text{null}(\Sigma) \geq (p+1)(n-p)$ with equality holding only when $M = S^p$ \cite[Theorem 5.1.1]{Simons}. For recent work discussing the problem of computing the index of minimal submanifolds in complex Einstein manifolds, we refer the reader to the preprint \cite{Kalafat:2024ctd}.

In the following, we consider the set of homogeneous hypersurfaces of the cohomogeneity-one Page and $Y^{p,q}$ metrics and identify those that are minimal. Our main results are to determine the index of these minimal hypersurfaces. 

\begin{theorem}\label{spectrum:Page}
The Page space $(\mathbb{CP}^2 \# \overline{\mathbb{CP}^2},g_P)$ admits a single minimal hypersurface within the set of surfaces of homogeneity that is a Berger sphere. It is totally geodesic, and the spectrum of the stability operator can be computed explicitly, and in particular, it has index 1 and nullity 0. 
\end{theorem} 

\begin{theorem}\label{Thm2}
The cohomogeneity-one Sasaki-Einstein spaces $Y^{p,q}$ on $S^2 \times S^3$ each admit a single minimal hypersurface within the set of surfaces of homogeneity. It is not totally geodesic and has topology $S^1 \times L(2,1)$. The spectrum of its stability operator can be computed explicitly and in particular, it has index 1.
\end{theorem} 
\begin{remark} To establish Theorem 1.2, the existence of the single negative eigenvalue of the stability operator is proved analytically.  We then prove that all other eigenvalues are positive apart from a small set of exceptional possibilities. The exclusion of additional negative eigenvalues arising from these remaining cases is verified numerically.
\end{remark}

The proofs of Theorem \ref{spectrum:Page} and Theorem \ref{Thm2} are given respectively in Section 2 and Section 3 below. We note that the problem of computing the index of minimal hypersurfaces for the Page metric was considered recently in \cite[Theorem 5.1]{Kalafat:2024ctd}; their result differs from ours in that it is claimed, based on their earlier work \cite{KalafatSari}, that there are additional homogeneous minimal surfaces in the Page space, whereas we find only a single such surface (see Proposition \ref{Minpage}).

\section{Minimal hypersurfaces in the Page metric}
\subsection{The Page space}
The Page space \cite{Page} is the manifold $M = \mathbb{CP}^2 \# \overline{\mathbb{CP}^2}$ equipped with an inhomogeneous Einstein metric $g_P$, given in local coordinates $(x,\psi,\theta,\phi)$ by
\begin{equation}
g = S \left[\frac{\td x^2}{A(x)} + 4 \alpha^2 A(x) \left(\td \psi + \frac{\cos\theta}{2} \td \phi \right)^2 + B(x) (\td \theta^2 + \sin^2\theta \td\phi^2) \right]
\end{equation} where 
\begin{equation}
A(x) = \frac{(3 - \nu^2 - \nu^2 (1+ \nu^2)x^2) (1-x^2) } {1 - \nu^2 x^2}, \qquad B(x) =  \frac{1 - \nu^2 x^2}{3 + 6\nu^2 - \nu^4}, \end{equation} and \begin{equation} S = \frac{3(1 + \nu^2)}{\Lambda}, \qquad  \alpha = (2 (3 + \nu^2))^{-1}.
\end{equation} Here  $x \in (-1,1)$, $\theta \in (0,\pi)$ and $\nu$ is a real parameter that is determined by the requirement that $g_P$ metric is Einstein, i.e. 
\begin{equation}
\text{Ric}(g_P) = \Lambda g_P.
\end{equation} This requires that $\nu$ be a positive root of the quartic
\begin{equation}\label{nu}
\nu^4 + 4\nu^3 - 6 \nu^2 + 12 \nu - 3 =0.
\end{equation} There is a unique such root $\nu \approx 0.281702$. Note $\det g = 4 S^4 \alpha^2 B(x) \sin^2\theta$. For convenience, we will normalize the volume of the Page metric by fixing $\Lambda = 3$ from now on; this is the scaling used for the unit round $n-$sphere $S^n$, namely, $\text{Ric}(g) = (n-1)g$. 

The metric extends to a globally smooth metric on the non-trivial $S^2$ bundle over $S^2$ provided the following identifications are satisfied: 
\begin{equation}\label{identifications}
(\psi, \phi) \sim (\psi + \pi, \phi + 2\pi), \qquad (\psi, \phi) \sim (\psi + 2\pi, \phi),
\end{equation}  The choice $\Lambda =3$ fixes the volume of the Page space to be 
\begin{equation}
\text{Vol}(g) = \frac{16 \pi^2 (3 - \nu^2)(1 + \nu^2)^2}{3 ( 3 + \nu^2)(3 + 6\nu^2 - \nu^4)} \approx 16\pi^2 \times 0.1060149
\end{equation} where we have used \eqref{nu}. 

 In this local chart, the $(\theta, \phi)$ coordinates parameterize the round $S^2$ base and the $(x,\psi)$ coordinates parameterize an $S^2$ with inhomogeneous metric. The Killing field $\partial_\psi$ degenerates smoothly at  $x = \pm 1$, the poles of the $S^2$ fibres. Surfaces of constant $x$ have $S^3$ topology with induced metric 
\begin{equation}
    h = S \left[4 \alpha^2 A(x) \left(\td \psi + \frac{\cos\theta}{2} \td \phi \right)^2 + B(x) (\td \theta^2 + \sin^2\theta \td\phi^2) \right].
\end{equation} This exhibits the form of a fibre metric of a circle bundle over $S^2$ with $\psi$ parameterizing the $S^1$ fibres.  The bundle is non-trivial with Chern number 1.

The metric is cohomogeneity-one and admits an $SU(2)\times U(1)$ isometry group generated by the Killing vector fields
\begin{equation}\label{SU(2)gen}
\begin{aligned}
R_1 &= -\cot\theta \cos\phi \partial_\phi - \sin \phi \partial_\theta + \frac{\cos\phi}{2\sin\theta} \partial_{\psi}, \\
R_2 & = -\cot\theta \sin \phi \partial_\phi + \cos\phi \partial_\theta + \frac{\sin\phi}{2\sin\theta} \partial_{\psi}, \\
R_3& = \partial_\phi, \qquad L_3 = \tfrac{1}{2}\partial_{\psi}.
\end{aligned}
\end{equation} and so $L_{R_i} g_P =0$, $i = 1,2,3$ and $L_{L_3} g_P =0$. The Killing fields satisfy the Lie algebra of $SU(2) \times U(1)$, namely, $[R_i, R_j] = -\epsilon_{ijk} R^k$, $[L_3, R_i] = 0$ with $\epsilon_{ijk}$ the totally antisymmetric tensor with $\epsilon_{123} = 1$. 

\subsection{Homogeneous minimal hypersurfaces}
The isometry group acts with three-dimensional orbits acting on the homogeneous $S^3$ level sets of the coordinate $x \in (-1,1)$. We are interested in classifying which of these hypersurfaces are minimal, i.e., critical points of the area functional. It is claimed in \cite{KalafatSari} that each such level set is minimal; we find, however, that only the `equatorial' level set at $x =0$ is minimal. The second fundamental form $K$ of a two-sided hypersurface $\Sigma \subset (M,g)$ with unit normal $n$ is given by the symmetric tensor field
\begin{equation}
    K(X,Y):=g(\nabla_X n, Y)
\end{equation} where $X,Y$ are vector fields tangent to $\Sigma$, $g$ is the ambient space metric, and $\nabla$ is the metric connection associated to $g$. While it is straightforward to compute the components of $K$ in the local coordinate basis, since the hypersurfaces being considered are just coordinate level sets, it is convenient to use the following `$n+1$ decomposition' of the ambient space metric
\begin{equation}\label{ADM}
    g = \beta^2 \td x^2 + h_{ij} (\td y^i + N^i \td x)(\td y^j + N^j \td x)
\end{equation} where $\beta$ is a function and $N$ is a vector field on the family of level sets $(\Sigma_x,h)$ (referred to as the `lapse' and `shift' in the relativity setting) and $y^i$ are local coordinates on $\Sigma_x$. Here $h_{ij}, N^i$ are understood to be functions of $x$, the parameter which labels the level sets.  It is a straightforward computation to determine the second fundamental form
(see e.g., \cite[Section 2.3]{Alcubierre}): 
\begin{equation}\label{2FF}
    K_{ij} = \frac{1}{2\beta} \left(\partial_x h_{ij} - D_i N_j - D_j N_i\right)
\end{equation} where $D$ is the metric connection associated to $h$. We have chosen the sign of $K$ so that the standard sphere in Euclidean space has a positive-definite second fundamental form with outward pointing normal $n$.
\begin{prop}\label{Minpage} Let $\Sigma_x$ be a homogeneous $S^3$ hypersurface defined as a level set of $x \in (-1,1)$. Then $\Sigma:=\Sigma_0$ is the only minimal hypersurface of this family. Moreover, it is totally geodesic. 
\end{prop}
\begin{proof} For the metric $g_P$, it is straightforward to read off 
\begin{equation}
    \beta = \left(\frac{S}{A(x)}\right)^{1/2}, \qquad N^i =0
\end{equation} The mean curvature of a level set hypersurface $\Sigma_x$ is 
\begin{equation}
    H_\Sigma = \text{Tr}_h K =  h^{ij} K_{ij} = \frac{1}{2\beta} h^{ij}\partial_x h_{ij} = \frac{1}{2\beta} \partial_x (\log \det h)
\end{equation} where
\begin{equation}
    \det h = S^3 \cdot 4 \alpha^2 A(x) B(x) \sin^2\theta.
\end{equation} This gives
\begin{equation}\label{HPage}
    H_\Sigma = \left(\frac{1}{S A(x)}\right)^{1/2} \frac{x (-3 - \nu^4  + 2\nu^2 x^2(1+\nu^2))}{1 - \nu^2 x^2}.
\end{equation} Recall $A(x) > 0$ in the domain $x \in (-1,1)$. It can be checked explicitly that the quadratic in round brackets in the numerator has two real roots which lie outside the domain. Thus $H_\Sigma = 0$ if and only if $x =0$. 

Next, observe that the components of the Page metric $g_{ij}$ and the induced metric $h_{ij}$ are functions of $x^2$ alone. It follows from \eqref{2FF} that each of the components $K_{ij}$ are odd functions of $x$ and hence they all vanish at $x=0$. Thus, the minimal hypersurface $\Sigma_0$ is totally geodesic. 
\end{proof}
We now turn to determining the index of the stability operator on $\Sigma_0$. 

\subsection{Proof of Theorem \ref{spectrum:Page}}
For the minimal hypersurface $\Sigma_0$ in the Page space, we may, in fact, explicitly determine the spectrum of $L$ by employing the homogeneity of the metric and demonstrating that it has index 1. 

Since $\text{Ric}(n,n) = 3 g(n,n) = 3$ we simply have
\begin{equation}\label{Lpage}
    L = -\Delta_\Sigma -3
\end{equation} The Laplacian with respect to $(\Sigma,h)$ is explicitly obtained from the formula
\begin{equation}
    \Delta_h f = \frac{1}{\sqrt{\det h}} \partial_i (\sqrt{\det h} h^{ij} \partial_j) f
\end{equation} To determine the spectrum of $L$, introduce the differential operator \cite{Wu:1976ge}
\begin{equation}
    D = \nabla_{S^2} - i n A
\end{equation} where $n \in \mathbb{Z}$ and $\nabla_{S^2}$ is the covariant derivative with respect to the round metric on the sphere with volume $\pi$, 
\begin{equation}
    \hat{g} = \frac{1}{4} (\td\theta^2 + \sin^2\theta \td\phi^2),
\end{equation} and we have defined the locally defined one-form
\begin{equation}
    A = \frac{\cos\theta}{2} \td\phi.
\end{equation}
A computation shows that
\begin{equation}
    D^2 = \hat{g}^{ij}D_i D_j = \frac{4}{\sin\theta} \partial_\theta (\sin\theta \partial_\theta) + \frac{4}{\sin^2\theta} \partial_\phi^2 - n^2 \cot^2\theta - 4 i n \frac{\cos\theta}{\sin^2\theta} \partial_\phi.
\end{equation} This can be thought of as a kind of generalization of the Laplacian on the round sphere. The spectrum of $D$ can be determined explicitly. In particular, the eigenfunctions $Y_{n,k}(\theta,\phi)$ satisfy
\begin{equation}
    D^2 Y_{n,k}(\theta, \phi) = \mu_{n,k} Y_{n,k}(\theta, \phi)
\end{equation} with $\mu_{n,k} \geq 0$ are a discrete family of eigenvalues. We will drop the eigenvalue labels for convenience. The values taken by $\mu$ are
\begin{equation}\label{mu}
    \mu = \ell(\ell + 2) - n^2, \; \text{ where }\; \ell = 2k + |n| \text{ with }\; k = 0,1,2,3 \ldots
\end{equation} To compute $\Delta_\Sigma$, we will show that it is the same as $D^2$ up to some additional terms. The inverse components of the metric are
\begin{equation}
    h^{\theta\theta} = \frac{1}{S B(x)}, \quad h^{\psi\psi} = \frac{1}{4\alpha^2 S A(x)} + \frac{\cot^2\theta}{4 S B(x)}, \quad h^{\psi\phi} = -\frac{\cos\theta}{2 B(x) \sin^2\theta}, \quad h^{\phi\phi} = \frac{1}{S B(x) \sin^2\theta}
\end{equation} which leads to 
\begin{equation}
\begin{aligned}
    \Delta_h f & = \frac{1}{S B(x)} \left[ \frac{1}{\sin\theta} \partial_\theta ( \sin\theta \partial_\theta f) + \frac{\partial_\phi^2 f}{\sin^2\theta}  \right] + \left( \frac{1}{4\alpha^2 S A(x)} + \frac{\cot^2\theta}{4 S B(x)}\right)\partial_{\psi}^2 f \\
    & - \frac{\cos \theta}{S B(x) \sin^2\theta} \partial_\psi \partial_\phi f
\end{aligned}
\end{equation} where $x$ is understood to be fixed (we are interested in the case $x=0$). Now assume that $f = f(\theta, \phi, \psi)$ is separable, that is $f = Y(\theta,\phi) e^{i n \psi}$ where $n \in \mathbb{Z}$ (recall $\psi$ has period $2\pi$). Substituting this into the equation $-\Delta_\Sigma f =  \lambda f$ is equivalent to the following equation for $Y(\theta,\phi)$:
\begin{equation}
    -D^2 Y = \left[ 4 S B(0) \lambda - \frac{n^2 B(0)}{\alpha^2 A(0)}\right] Y
\end{equation} Then we can take as eigenfunctions the charged scalar harmonics $Y(\theta, \phi)$ introduced above with the shifted eigenvalues
\begin{equation}
    \lambda = \frac{\mu}{4S B(0)} + \frac{n^2}{4 S \alpha^2 A(0)}
\end{equation} where $\mu$ is given by \eqref{mu}, which explicitly is given by
\begin{equation}
    \mu_{n,k} = (4k + 2)|n| + 4k(k +1), \qquad n \in \mathbb{Z},\; k \in \mathbb{N} \cup \{0 \}.
\end{equation} Note that $\mu_{1,0} = 2, \mu_{2,0} = 4, \mu_{3,0} = 6, \mu_{4,0} = \mu_{0,1} = 8$ are the lowest values of $\mu$. The corresponding lowest non-zero eigenvalues are
\begin{equation}
    \lambda_{1,0} =4.61536, \quad \lambda_{2,0} = 15.2467, \quad \lambda_{0,1} = 6.42946
\end{equation} The eigenvalues $\lambda^L_{n,k}$ of the stability operator \eqref{Lpage} are then given by 
\begin{equation}
    \lambda^L_{n,k} = \lambda_{n,k} - 3
\end{equation} and are all positive apart from a single negative eigenvalue $L$ corresponding to $\mu_{0,0} =0$. It follows that $\text{index}(\Sigma) = 1$. This completed the proof of Theorem \ref{spectrum:Page}.

\section{Minimal hypersurfaces in the Sasaki-Einstein spaces $Y^{p,q}$}
\subsection{$Y^{p,q}$ Sasaki-Einstein metrics}
An interesting class of five-dimensional cohomogeneity-one Einstein manifolds which admit homogeneous minimal (but not totally geodesic) hypersurfaces are the countably infinite family of Sasaki-Einstein manifolds $Y^{p,q} \cong S^2 \times S^3$ explicitly constructed in \cite{Gauntlett:2004yd}. Recall that a Sasaki-Einstein manifold is an Einstein manifold whose metric cone is Ricci-flat and K\"ahler.  The metric can be expressed in local coordinates as 
\begin{equation}\label{SE}
\begin{aligned}
    g &= \frac{1-y}{6} (\td \theta^2 + \sin^2\theta \td \phi^2) + \frac{\td y^2}{w(y) q(y)} + \frac{q(y)}{9} (\td \psi - \cos\theta \td \phi)^2 \\
    & + w(y) \left[ \td \alpha + f(y)(\td \psi - \cos\theta \td \phi)\right]^2
\end{aligned}
\end{equation} with
\begin{equation}
    w(y) = \frac{2(b-y^2)}{1-y}, \quad q(y) = \frac{b - 3 y^2 + 2 y^3}{b-y^2}, \quad f(y) = \frac{b - 2y + y^2}{6(b-y^2)}
\end{equation} and $b$ is a constant. This is an Einstein metric satisfying
\begin{equation}
    R_{ab} = 4 g_{ab}.
\end{equation} An extensive global analysis of this family of metrics is given in \cite{Gauntlett:2004yd}. Let $p,q$ be coprime positive integers with $ p > q$. Provided that 
\begin{equation}
    b = \frac{1}{2} - \frac{p^2 - 3q^2}{4p^3} \sqrt{4p^2 - 3 q^2}.
\end{equation} and 
\begin{equation}\label{ell}
    \ell = \frac{q}{3 q^2 - 2 p^2 + p \sqrt{4p^2 - 3 q^2}},
\end{equation} the local metrics \eqref{SE} globally extend to a countably infinite family of  smooth metrics on $S^2 \times S^3$ characterized by $(p,q)$, provided that we take the ranges of the coordinates to be
\begin{equation}
    y_1 \leq y \leq y_2, \quad 0 \leq \theta \leq \pi, 
\end{equation} where the endpoints $y_1, y_2$ are the two smallest roots of the cubic appearing in the numerator of $q(y)$, that is, $b - 3 y^2 + 2y^3$:
\begin{equation}
    y_{1,2} = \frac{1}{4p} \left( 2p \pm 3q - \sqrt{4p^2 - 3 q^2} \right).
\end{equation} Note that setting $\epsilon = q/p \in \mathbb{Q}$ we can express
\begin{equation}\label{bpq}
    b = \frac{1}{2} - \frac{1 - 3 \epsilon^2}{4} \sqrt{4 - 3\epsilon^2}
\end{equation} and it is elementary to see that $b \in (0,1)$ since $\epsilon \in (0,1)$. Furthermore, $b - y_i^2 > 0$ so that the quantity $b - y^2 > 0$ everywhere for $y \in (y_1,y_2)$. Similar reasoning shows that 
$y_1 < 0 < y_2$ so that $y=0$ lies in the allowed domain of $y$, and $y_1 \in (-1/2,0)$ and $y_2 \in (0,1)$. Finally, the coordinates $(\phi,\psi,\alpha)$ are subject to the identifications 
\begin{equation}\label{ident}
(\psi,\phi,\alpha) \sim (\psi + 2\pi, \phi, \alpha), \quad (\psi,\phi,\alpha) \sim (\psi + 2\pi, \phi + 2\pi, \alpha), \quad  (\psi,\phi,\alpha) \sim (\psi, \phi, \alpha + 2\pi \ell).
\end{equation} Note that the coordinate $\phi$ is not independently $2\pi-$periodic, but must be accompanied by a half-rotation along the $\psi$ direction. The metric is cohomogeneity-one with isometry group $SU(2)_{\times \mathbb{Z}_2} \times U(1) \times U(1)$ \cite{Gauntlett:2004yd}. In particular, the surfaces of homogeneity are the level sets $\Sigma_y$ defined by $y= y_0 \in (y_1, y_2)$, which have induced a homogeneous metric
\begin{equation}
     h = \frac{1-y}{6} (\td \theta^2 + \sin^2\theta \td \phi^2) + \frac{q(y)}{9} (\td \psi - \cos\theta \td \phi)^2 + w(y) \left[ \td \alpha + f(y)(\td \psi - \cos\theta \td \phi)\right]^2.
\end{equation} For convenience, we will simply refer to these level sets as $\Sigma$ hereafter. The Killing vector fields $\partial/ \partial \alpha, \partial / \partial \psi$ generate the $U(1) \times U(1)$ subgroup of the isometry group, whereas the $SU(2)$ acts with three-dimensional orbits on the $(\psi,\theta,\phi)$ with generators \eqref{SU(2)gen} with the replacement $\phi \to -\phi$). These  hypersurfaces have topology $S^1 \times L(2,1)$. One way to see this is to appeal to the classification of closed (connected) orientable Riemannian manifolds $M$ of dimension $s$ admitting a smooth effective torus $\mathbb{T}^{s-1}$ action as isometries given in~\cite[Theorem 2]{Hollands}.  This theorem states that the orbit space $M / \mathbb{T}^{s-1}$ is either a closed interval or a circle, and furthermore, in the former case, we have that $M$ is diffeomorphic to (i) $S^2 \times \mathbb{T}^{s-2}$, (ii) $S^3 \times \mathbb{T}^{s-3}$, or (iii)  $L(p,q) \times \mathbb{T}^{s-3}$.  These possibilities are characterized by specifying which integer linear combination of $2\pi-$periodic generators of the torus action degenerate at each boundary point.  In particular, if $\eta_i$, $i = 1,2,\ldots s-1$ denote a basis of generators, one may always choose $\eta_1$ to vanish at one endpoint and $q \eta_1  + p \eta_2$ with $p,q \in \mathbb{Z}$ at the other. The case $(p,q) = (0,1)$ corresponds to (i), $(p,q) = (1,1)$ to (ii), and otherwise we have case (iii). 

In the setting of interest here, $s=4$ and the orbit space $M / \mathbb{T}^3$ (parameterized by the coordinate $\theta \in [0,\pi]$) is diffeomorphic to a closed interval.  To verify the above claim on the topology of $\Sigma$, observe that the Killing fields 
\begin{equation}
    K_1 = \frac{\partial}{\partial \psi} + \frac{\partial}{\partial \phi}, \qquad K_2 = -\frac{\partial}{\partial \psi} + \frac{\partial}{\partial \phi} , \qquad K_3 = \ell \frac{\partial}{\partial \alpha} , \qquad K_4 = \frac{\partial}{\partial \psi}
\end{equation} each generate a $2\pi-$periodic flow with $K_1$ and $K_2$ normalized to degenerate smoothly at $\theta =0$ and $\theta = \pi$ respectively.   Observe that $K_1 = 2K_4 + K_2$, so we may choose $\eta_1 = K_2, \eta_2 = K_4, \eta_3 = K_3$ as a basis for generators of the $\mathbb{T}^3$ action. Then $\eta_1$ vanishes at one endpoint and $2\eta_2 + \eta_1$ vanishes at the other.  This implies \cite[Theorem 2]{Hollands} that $\Sigma$ has topology $L(2,1) \times S^1$ .
\subsection{Homogeneous minimal hypersurfaces} For a hypersurface defined as a level set of $y$, it is straightforward to compute the second fundamental form using the local formula \eqref{2FF} where now $N^i =0$ (since $\partial/\partial y$ is hypersurface orthogonal) and $\beta = \sqrt{g_{yy}}$:
\begin{equation}
    K_{ij} = \frac{1}{2\sqrt{g_{yy}}} \partial_y h_{ij}, \qquad g_{yy} = \frac{1-y}{2Q(y)}
\end{equation} where $i,j = \{\theta,\psi,\phi,\alpha\}$ and $Q(y):=b - 3y^2 + 2y^3$. 
\begin{prop}\label{minhyp}  There is a single minimal hypersurface of the above form, corresponding to a level set $y = y_m$ where $y_m \in (y_1,0)$. 
\end{prop}
\begin{proof} The determinant of the induced metric $h_{ij}$ is given by
\begin{equation}
\det h = \frac{(1-y)^2}{324} q(y) w(y) \sin^2\theta.
\end{equation} The mean curvature of the hypersurface is then 
\begin{equation}
    H_\Sigma = \text{Tr}_h K = h^{ij}K_{ij} = \frac{1}{2\sqrt{g_{yy}}} \partial_y (\log \det h).
\end{equation} Straightforward computation shows that the condition $H_\Sigma=0$ is equivalent to solving the cubic
\begin{equation}
    H(y):=b + 6 y - 15 y^2 + 8 y^3 =0
\end{equation}  The discriminant of $H(y)$ is 
\begin{equation}
    -108 \left(16 b^2-5b - 11\right) = 108 \cdot  \left(16 b + 11\right)(1-b)
\end{equation} In the range $0 < b < 1$, this quantity is strictly positive. Thus, there are three distinct real roots $y^*_i$, $i=1,2,3$. It remains to determine which of these roots lie in the range $(y_1, y_2)$ and thus constitute a minimal surface. From Descartes' rule of signs, $H(y)$ can have two or zero positive roots. It is easy to check that the cubic $H(y)$ has a local maximum at $y = 1/4$ and a local minimum at $y = 1$. In addition $H(0) >0$, $H(1/4) > 0$, and $H(1) = b-1 < 0$. It follows that $H(y)$ must have positive roots $y_2^*, y_3^*$ at some $1/4 < y_2^* < 1$, $y_3^* > 1$ and a single negative root $y_1^*$. Clearly $y_3^*$ cannot lie inside $(y_1,y_2)$. We can also see that the cubic $Q(y)$ has a local maximum at $y = 0$ and a local minimum at $y = 1$. Thus (as expected) it will be positive between one negative root $y_1 <0$ and one positive root $y_2 \in (0,1)$, with a third positive root $y_3>1$. 

At any root $y_i^*$ of $H(y)$,  $Q(y) = b - 3y^2 + 2y^3 = -6 y_i^{*}(1 - y_i^{*})^2$. Observe that for $y_1^* < 0$, this cubic is positive; but for $y^*_2 > 0$ it will be negative. In particular, $g_{yy} < 0$ and the induced metric $h$ is not positive definite (since $q(y^*_2) < 0$). It follows that only the negative root $\bar{y}:=y_1^* \in (y_1,y_2)$ of $H(y)$ gives rise to a minimal surface. 
\end{proof}

This minimal surface is not totally geodesic. Let $\tilde{K}_{ij}:=\sqrt{g_{yy}} K_{ij}$. Direct computation shows 
\begin{equation}
    \begin{aligned}
        \tilde{K}_{\alpha\alpha} &= - 8\bar{y}, \qquad \tilde{K}_{\alpha\psi} = -\frac{1 + 4\bar{y}}{4}, \qquad \tilde{K}_{\alpha\phi} = \tilde{K}_{\alpha\psi} \cos\theta \\
        \tilde{K}_{\theta\theta} &= -\frac{1}{12}, \qquad \tilde{K}_{\psi\psi} = -\frac{1 + 2\bar{y}}{9}, \qquad \tilde{K}_{\psi\phi} = \tilde{K}_{\psi\psi} \cos\theta, \\ \tilde{K}_{\phi\phi} &= -\frac{1}{72} \left(7 + \cos(2\theta) + 8\bar{y} (1 + \cos2\theta)\right)
    \end{aligned}
\end{equation} where we have explicitly eliminated $b$ using $H(\bar{y}) =0$. The inverse metric components can be read off from 
\begin{equation}
    h^{-1} = h^{ij}\partial_i \partial_j = \frac{1}{w(y)} (\partial_\alpha)^2 + \frac{9}{q(y)} (\partial_\psi - f\partial_\alpha)^2 + \frac{6}{1-y} \left(\partial_\theta^2 + \csc^2\theta \left(\partial_\phi + \cos\theta \partial_\psi\right)^2\right).
\end{equation} We then compute
\begin{equation}
    \text{Tr}(K^2) = K^{ij}K_{ij} = \frac{2(1 - 10\bar{y})}{1 - \bar{y}}.
\end{equation} It then follows that
\begin{equation}
    \text{Ric}(n,n) + K^{ij}K_{ij} = \frac{6(1 - 4\bar{y})}{1 - \bar{y}} > 0
\end{equation} using the Einstein condition $\text{Ric}(n,n) = 4$. The stability operator is thus given by
\begin{equation}
    L = -\Delta_\Sigma - \frac{6(1 - 4\bar{y})}{1 - \bar{y}}.
\end{equation} The set of eigenvalues $\tilde\lambda_i$ of $L$ is therefore given by the eigenvalues $\lambda_i$ of the Laplacian up to a constant shift. It is clear that there is always at least one negative eigenvalue of $L$ with a constant eigenfunction.  

\subsection{Proof of Theorem \ref{Thm2}} We can obtain the eigenvalues $\tilde\lambda_i$ by reducing the problem to an ordinary differential equation, which can be solved in terms of hypergeometric functions. The Laplacian is given by
\begin{equation}
    \begin{aligned}
        \Delta_\Sigma &= \frac{6}{1-\bar{y}} \left[ \frac{1}{\sin\theta} \partial_\theta (\sin\theta \partial_\theta) + \frac{1}{\sin^2\theta} \left(\partial_\phi + \cos\theta \partial_\psi\right)^2 \right] + \frac{\partial_\alpha^2}{w(\bar{y})} + \frac{9 f(\bar{y})^2}{q(\bar{y})} \left(\partial_\alpha - \frac{\partial_\psi}{f(\bar{y})}\right)^2.
    \end{aligned}
\end{equation}  Due to the periodicities \eqref{ident}, we consider solutions of $\Delta u = -\lambda u$ which take the separable form
\begin{equation}
    u = U(\theta) e^{ i N_\psi \psi} e^{i N_\phi \phi} e^{i N_\alpha \alpha / \ell}
\end{equation} where $N_\psi, N_\phi, N_\alpha \in \mathbb{Z}$. Substituting this into the eigenvalue equation gives
\begin{equation}
\begin{aligned}
    & \frac{6}{1-\bar{y}} \left[ \frac{1}{U \sin\theta} \partial_\theta ( \sin\theta \partial_\theta U) - \frac{1}{\sin^2\theta} (N_\phi + \cos\theta N_\psi)^2 \right] \\ & - \frac{N_\alpha^2}{\ell^2 w(\bar{y)} } - \frac{9f(\bar{y})^2}{q(\bar{y})}\left(\frac{N_\alpha}{\ell} - \frac{N_\psi}{f(\bar{y})}\right)^2 = -\lambda .
    \end{aligned}
\end{equation} The first term above in the square brackets must itself be a constant in order for this equation to hold (since every other term is a constant). This implies
\begin{equation}\label{eigen:Lap}
    \lambda = \frac{N_\alpha^2}{\ell^2 w(\bar{y})} + \frac{9f(\bar{y})^2}{q(\bar{y})}\left(\frac{N_\alpha}{\ell} - \frac{N_\psi}{f(\bar{y})}\right)^2 + \frac{6(E - N_\psi^2)}{1-\bar{y}}
\end{equation} where $E$ is a constant defined by the eigenvalue equation
\begin{equation}\label{auxeigen}
    \left[ -\frac{1}{U \sin\theta} \partial_\theta ( \sin\theta \partial_\theta U) + \frac{1}{\sin^2\theta} (N_\phi + \cos\theta N_\psi)^2 + N_\psi^2\right] = E.
\end{equation} This auxiliary eigenvalue problem can be solved using the approach of \cite{Gibbons:2004em}. Define $\vartheta = \theta/2$ with $\vartheta \in (0,\pi/2)$ since $\theta \in (0,\pi)$. Then \eqref{auxeigen} becomes
\begin{equation}
-\frac{1}{4 \sin\vartheta \cos\vartheta} \partial_\vartheta( \sin\vartheta \cos\vartheta \partial_\vartheta U) + \frac{1}{4 \sin^2\vartheta\cos^2\vartheta} \left[ N_\phi + \cos2\vartheta N_\psi\right]^2 U + N_\psi^2 U = E U
\end{equation} Next define integers $n_\psi, n_\phi$ by
\begin{equation}
    n_\psi = N_\psi - N_\phi, \qquad n_\phi = N_\psi + N_\phi.
\end{equation} Then we get
\begin{equation}
-\frac{1}{ \sin\vartheta \cos\vartheta} \partial_\vartheta( \sin\vartheta \cos\vartheta \partial_\vartheta U) + \left(\frac{(n_\psi)^2}{\cos^2\vartheta} + \frac{(n_\phi)^2}{\sin^2\vartheta} \right) U= 4E U
\end{equation} This is the same as \cite[Eq.(36)]{Gibbons:2004em}, with the identifications $A = n_\phi$, $B = n_\psi$, where it is shown to possess three regular singular points. The two linearly independent solutions can be written in terms of hypergeometrics:
\begin{equation} \nonumber
    u_\pm = \left( \cos^{\pm B} \vartheta \sin^A\vartheta\right) \tensor[_2]{F}{_1} \left( \frac{1 + A \pm B - (1 + 4E)^{1/2}}{2}, \frac{1 + A \pm B + (1 + 4E)^{1/2}}{2}, 1 \pm B; \cos^2\vartheta \right).
\end{equation} Regularity as $\vartheta \to \pi/2$ requires that $\pm B \geq 0$ for $u_\pm$ while regularity as $\vartheta \to 0$ for both $u_\pm$ implies $A \geq 0$ \cite{Gibbons:2004em}. Furthermore, the hypergeometrics have a singular point when the argument $\cos^2\vartheta \to 1$ and regularity requires that they truncate to polynomials. This fixes the conditions
\begin{equation}\label{reg:k}
\frac{1 + |A| + |B| \pm (1 + 4E)^{1/2}}{2} = -k, \qquad k = \mathbb{N} \cup \{0 \}.
\end{equation} Set $E = J(J+1)$. Then \eqref{reg:k} can be solved by setting
\begin{equation}
    J = k + \frac{1}{2}( |n_\psi| + |n_\phi|).
\end{equation} Note that $J \geq 0$ although it can take half-integer values. 
Also note that 
\begin{equation}
    E - N_\psi^2 = k(k+1) + \left( k + \frac{1}{2}\right)(|n_\psi| + |n_\phi|) + \frac{|n_\psi| |n_\phi| - n_\psi n_\phi}{2} \geq 0
\end{equation} since each term on the right-hand side is non-negative. It then follows from \eqref{eigen:Lap} that the set of eigenvalues $\lambda = \lambda_\gamma$, where the label $\gamma =\{k,N_\alpha,N_\psi,N_\phi\}$, are non-negative  and vanish if and only if $k = N_\alpha = N_\psi = N_\phi=0$. The eigenvalues of the stability operator $L$ are then
\begin{equation}\label{eigen:L}
    \tilde{\lambda}_\gamma = \frac{N_\alpha^2}{\ell^2 w(\bar{y})} + \frac{9f(\bar{y})^2}{q(\bar{y})}\left(\frac{N_\alpha}{\ell} - \frac{N_\psi}{f(\bar{y})}\right)^2 + \frac{6(E - N_\psi^2 - 1 + 4\bar{y})}{1-\bar{y}} .
\end{equation} Clearly, when $\gamma = (0,0,0,0)$ $\tilde\lambda_\gamma < 0$. Our goal is to determine whether this exhausts the set of negative eigenvalues. 

We have shown that $-1 < 
\bar{y} < 0$. By expressing $b$ as a function of $\epsilon = q/p \in (0,1)$ using \eqref{bpq}, we find that $\bar{y}$ is the single negative root of the cubic
\begin{equation}
    8 y^3 - 15 y^2 + 6y + \frac{1}{2} - \frac{1 - 3\epsilon^2}{4} \sqrt{4 - 3 \epsilon^2} =0
\end{equation} Note that when $\epsilon=1$ this cubic can be factorized explicitly
\begin{equation}
    (1 + 8 y)(1- y)^2 =0
\end{equation} 
\begin{figure}[h]
    \centering
    \includegraphics[width=0.5\linewidth]{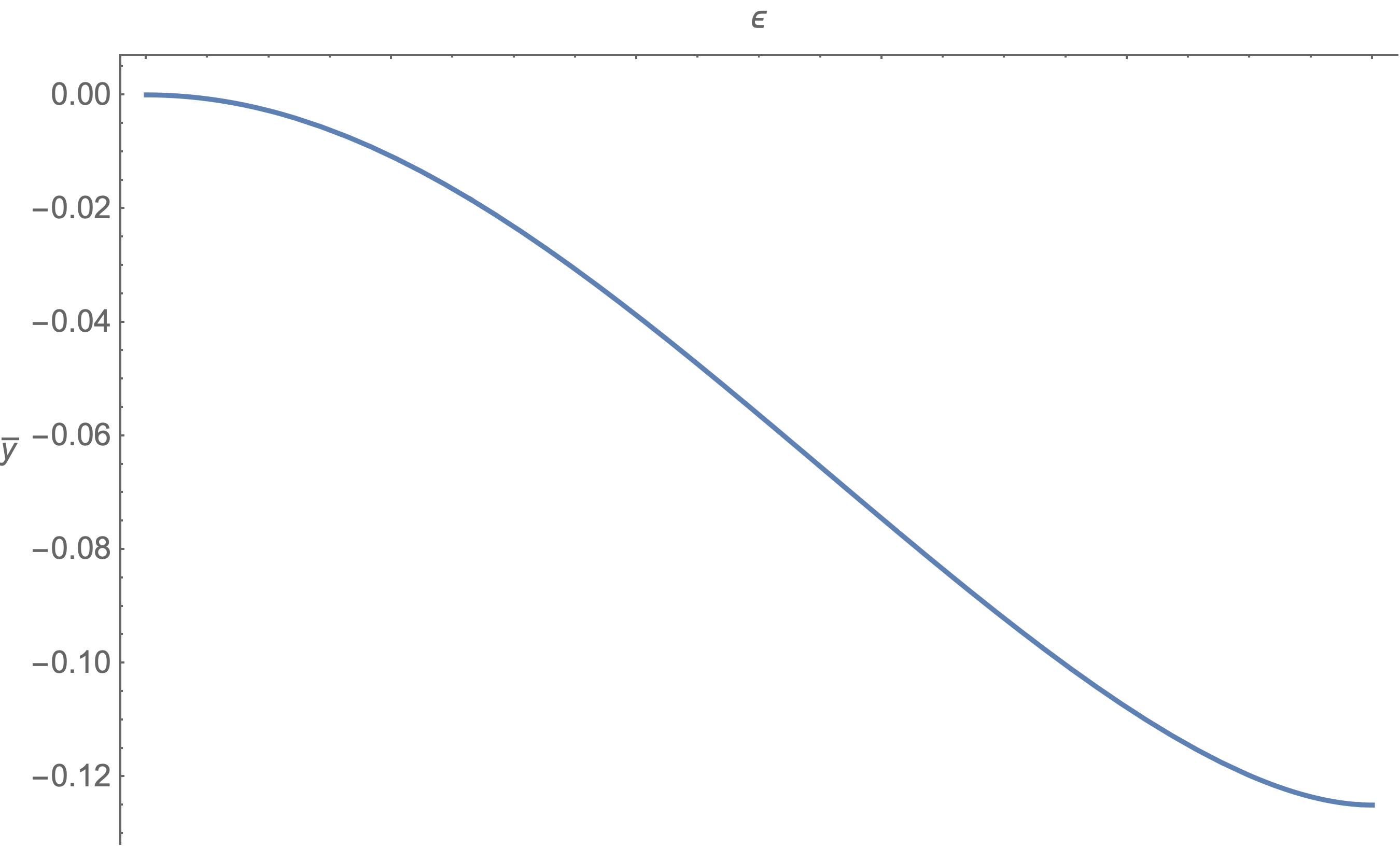}
    \caption{$\bar{y}$ as a function of $\epsilon = q/p \in (0,1)$. As $\epsilon \to 1$, $\bar{y} \to -1/8$.} 
    \label{negeigenvalue}
\end{figure} so that $\bar{y} = -1/8$. Using the explicit formulas for the roots of a cubic, one can argue that $\bar{y}$ is a monotonically decreasing function of $\epsilon \in (0,1)$ (this can be easily seen by plotting $\bar{y}(\epsilon)$ as shown in Figure \ref{negeigenvalue}). Hence $\bar{y} > -1/8$. Then provided $k \geq 1$, 
\begin{equation}
    E - N_\psi^2 - 1 + 4\bar{y} >k(k+1) + \left( k + \frac{1}{2}\right)(|n_\psi| + |n_\phi|) + \frac{|n_\psi| |n_\phi| - n_\psi n_\phi}{2}  -\frac{3}{2} \geq \frac{1}{2}
\end{equation} and hence $\tilde\lambda_\gamma >0$. Thus negative eigenvalues are only possible if $k=0$ which we will now assume.  Suppose $k= N_\alpha =0$ and $(N_\psi,N_\phi) \neq (0,0)$. Then 
\begin{equation}
     E - N_\psi^2 - 1 + 4\bar{y} > \frac{1}{2}\left(|n_\psi| + |n_\phi| + |n_\psi| |n_\phi| - n_\psi n_\phi \right)  - \frac{3}{2}
\end{equation} It is easy to check that the right hand side (and therefore $\tilde\lambda_\gamma$) will be positive unless $(N_\psi,N_\phi) \in \mathcal{S} = \{ (\pm 1,0), (\pm 1,\pm 1)\}$. For each of the exceptional cases in $\mathcal{S}$,
\begin{equation}
     \frac{1}{2}\left(|n_\psi| + |n_\phi| + |n_\psi| |n_\phi| - n_\psi n_\phi \right) = 1.
\end{equation} Moreover, including $N_\alpha \neq 0$ will only contribute positive terms to $\tilde\lambda_\gamma$, so we automatically have positivity for any $\gamma = (0,N_\alpha, N_\psi,N_\phi)$ provided $(N_\psi, N_\phi) \notin \mathcal{S}$. As a simple illustrative example, consider the case $(N_\psi,N_\phi) = (0,\pm 1)$. Observe that every other term in $\tilde\lambda_{(0,0,0,\pm 1)}$ vanishes apart from the last in \eqref{eigen:L}.  We have $n_\phi = -n_\psi = 1$ and hence $J=1$, and $E=2$.  $E - N_\psi^2 - 1 + 4\bar{y}= 1 + 4\bar{y}$. Hence 
\begin{equation}
\tilde\lambda_{(0,0,0,\pm 1)} = \frac{6 (1 + 4 \bar{y}) }{1 - \bar{y}} > 0.
\end{equation} since $-1/8 < \bar{y} < 0$. 

We will examine whether the set $\mathcal{S}$ yields any negative eigenvalues shortly. First, however, we deal with the case $\gamma = (0,N_\alpha,0,0)$ where $N_\alpha \in \mathbb{Z}$, $N_\alpha \neq 0$. We can argue that $\tilde\lambda_\gamma > 0$ as follows. Since 
\begin{equation}\label{Nalpha}
       \tilde{\lambda}_{(0,N_\alpha,0,0)} = \frac{N_\alpha^2}{\ell^2 w(\bar{y})} + \frac{9 f(\bar{y})^2 N_\alpha^2}{\ell^2 q(\bar{y})} - \frac{6(1 - 4\bar{y})}{1-\bar{y}} .
\end{equation} It is straightforward to show that 
\begin{equation}\label{boundshift}
   -8< - \frac{6(1 - 4\bar{y})}{1-\bar{y}} < -6
\end{equation} for $\bar{y} \in (-1/8,0)$.  Meanwhile, from the definition \eqref{ell} and the bound $\epsilon \in (0,1)$ it is straightforward to establish
\begin{equation}
4 q^2 < \frac{1}{\ell^2} < \frac{81 q^2}{16}.
\end{equation} Similarly, the function
\begin{equation}
J(\bar{y}) := \frac{1}{w(\bar{y})} + \frac{9 f(\bar{y})^2}{q(\bar{y})} = \frac{1}{3} - \frac{1}{12\bar{y}}
\end{equation} is monotonically increasing and hence $J(\bar{y}) > J(-1/8) = 1$. Using \eqref{boundshift} we then have
\begin{equation}
  \tilde{\lambda}_{(0,N_\alpha,0,0)}  > \frac{N_\alpha^2}{\ell^2} J(\bar{y}) - 8 >  4 N_\alpha^2 q^2 - 8.
\end{equation} Since $q \in \mathbb{N}$, it follows immediately that $\tilde\lambda_{(0,N_\alpha,0,0)} > 0$ unless $(|N_\alpha|, q) = (1, 1)$.  For this case, we have found it necessary to use the explicit formula for the negative cubic root $\bar{y}$.  Let $u: = \sqrt{4 - 3 \epsilon^2}$ and observe $\epsilon \in (0,1)$ implies $u \in (1,2)$.  Using the standard formula for roots of a cubic, we obtain
 \begin{equation}\label{exproot}
 \bar{y} = \frac{1}{8} \left[ 5 - 3 \cos\left(\frac{\varphi}{3}\right) - 3 \sqrt{3} \sin \left(\frac{\varphi}{3} \right) \right].
 \end{equation} Here 
 \begin{equation}
\varphi = \mathrm{arg} \, z , \qquad z =  (1 + 2u)(4 u^2 - 2u - 11) + 4(u-1) \sqrt{(u+2)(19 + 12 u - 4 u^2)} i 
 \end{equation}  which defines $\varphi$  up to addition of $2\pi m$, $m = 0,1,2$. The choice of $m$ is fixed so that $\bar{y}$ is the unique negative root as guaranteed by Proposition \ref{minhyp}.  Note that in the limiting case $\epsilon  =u =1$, we have $\varphi = \pi + 2\pi m$. The choice $m=0$ produces $\bar{y} = -1/8$ as expected.  Using \eqref{exproot}, and the expressions for $w(\bar{y}), q(\bar{y}), f(\bar{y})$,  one can express \eqref{Nalpha} as a complicated function of the parameter $\epsilon$ only.  Note that since $p > q$, for $q=1$,  $\epsilon \in (0,1/2)$.  Examining the resulting expression for $\tilde\lambda_\gamma (\epsilon)$  one can verify the function is continuous and diverges as $\epsilon \to 0$.  Using {\sc Mathematica} it is straightforward to plot this function and verify by inspection that it is strictly positive in the allowed range of $\epsilon$ (see Figure \ref{fig:main}). Although it is not obvious, $\tilde\lambda$ only becomes negative for $\epsilon > \epsilon_* \approx 0.584142$ which lies outside the allowed range. 
\begin{figure}[htbp]
        \centering
        \includegraphics[width=9cm]{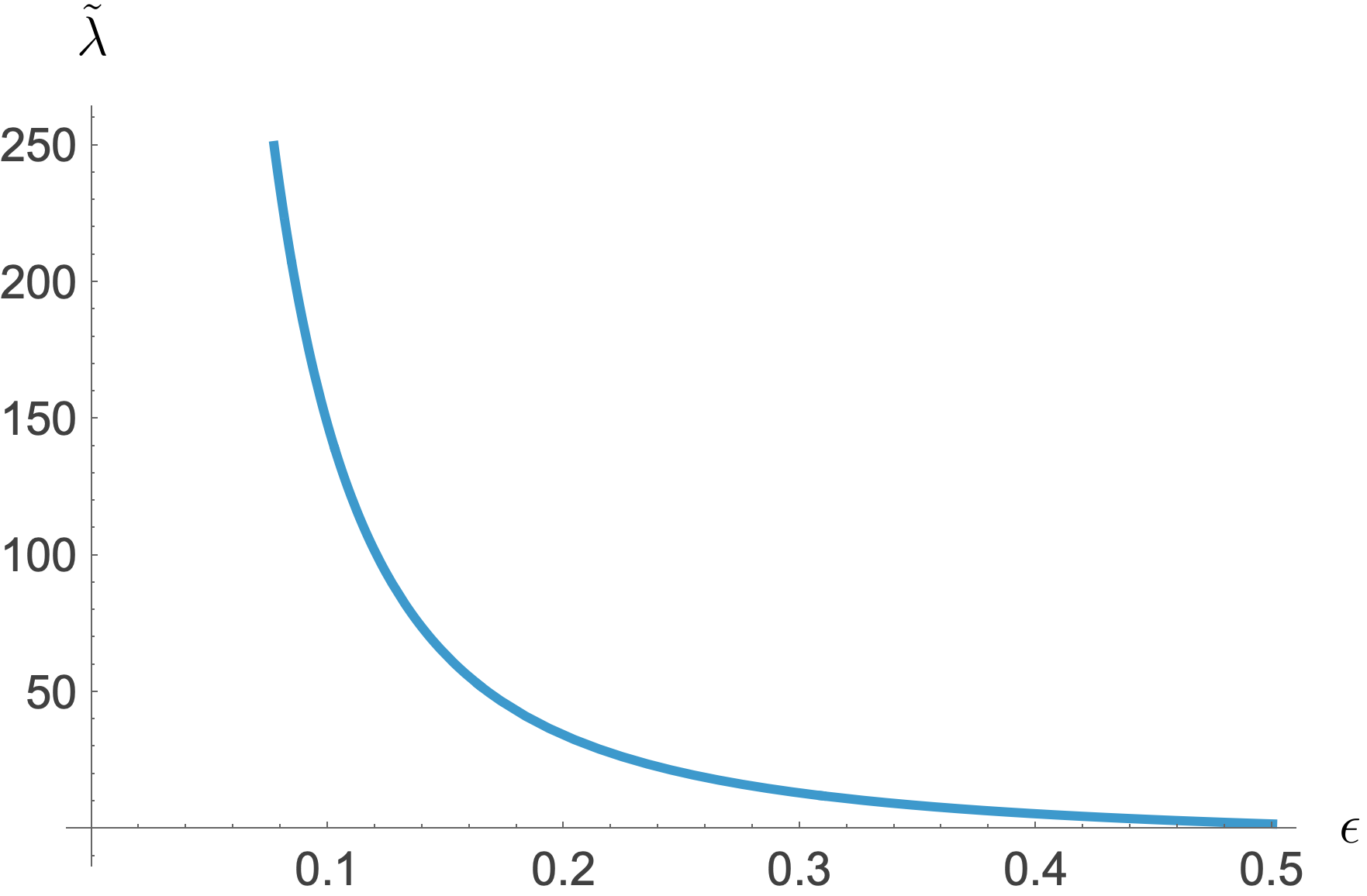}
    \caption{$\tilde\lambda$ for $(N_\alpha,q) = (\pm 1,1)$.} 
    \label{fig:main}
\end{figure} 

Moreover, consider any eigenvalue with $\gamma = (0,N_\alpha,0,N_\phi)$ where $N_\phi \neq 0$. Since $E = |N_\phi| (|N_\phi| + 1)$, using 
\begin{equation}
\frac{6 E}{1 - \bar{y}}> \frac{16 E}{3} \geq \frac{32}{3} 
\end{equation} from \eqref{eigen:L} and the discussion above, we must have $\tilde\lambda_{(0,N_\alpha,0,N_\phi)}  > 0$ and hence there are no further negative eigenvalues in this subset. 

It remains to check the cases (i)  $\gamma = (0,0,\pm 1,0)$, $(0,0, \pm 1,\pm 1)$ for any $q$, and  (ii) $\gamma = (0,N_\alpha,\pm 1,0)$ or $(0,N_\alpha,\pm 1,\pm 1)$ for $N_\alpha = \pm 1$ and $q=1$.  In case (i) , we have
\begin{equation}
\tilde\lambda_{(0,0,\pm 1,0)} = \tilde\lambda_{(0,0,\pm 1,\pm 1)}  = \frac{9}{q(\bar{y})} + \frac{24 \bar{y}}{1 - \bar{y}}  = \frac{3(3 + 4\bar{y})}{1 - \bar{y}} > \frac{20}{3}
\end{equation} where we used the lower bound on $\bar{y}$ and monotonicity of the function on the left hand side of the inequality.  Finally, for case (ii) there are only two independent possibilities, namely $N_\alpha = -N_\psi = \pm 1$ and $N_\alpha = N_\psi = \pm 1$.  We must evaluate
\begin{equation}
    \ \frac{1}{\ell^2 w(\bar{y})} + \frac{9f(\bar{y})^2}{q(\bar{y})}\left(\frac{1}{\ell} \mp  \frac{1}{f(\bar{y})}\right)^2 + \frac{24\bar{y}}{1-\bar{y}} .
\end{equation} These result in complicated, albeit continuous, functions of $\epsilon$ once we use the explicit form of the cubic root $\bar{y}$ \eqref{exproot} and $\ell$ \eqref{ell}.  It is straightforward to verify by inspection using {\sc Mathematica} to plot these functions that they are strictly positive for $\epsilon \in (0,1)$ and in particular for $\epsilon < 1/2$ (recall $q=1, p >1$). The resulting plots are given in Figure \ref{fig3} below.  We conclude that there should be no more negative eigenvalues in this subset. 
\begin{figure}[htbp]
    \begin{subfigure}[b]{0.48\textwidth}
        \centering
        \includegraphics[width=\linewidth]{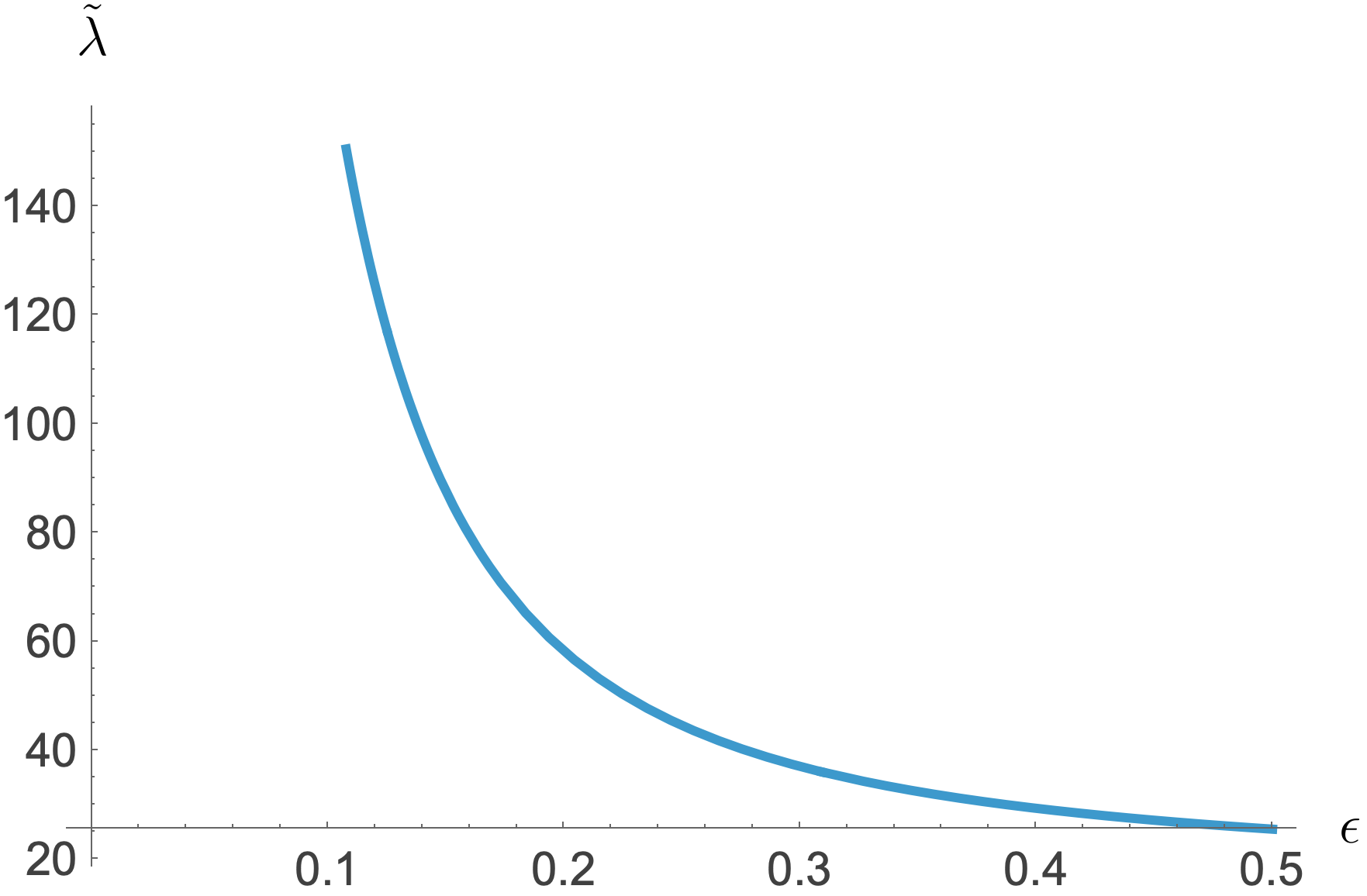}
        \caption{$(N_\alpha,N_\psi) = (1,-1), (-1,1)$.}
        \label{fig:sub2}
    \end{subfigure}
    \hfill
        \begin{subfigure}[b]{0.48\textwidth}
        \centering
        \includegraphics[width=\linewidth]{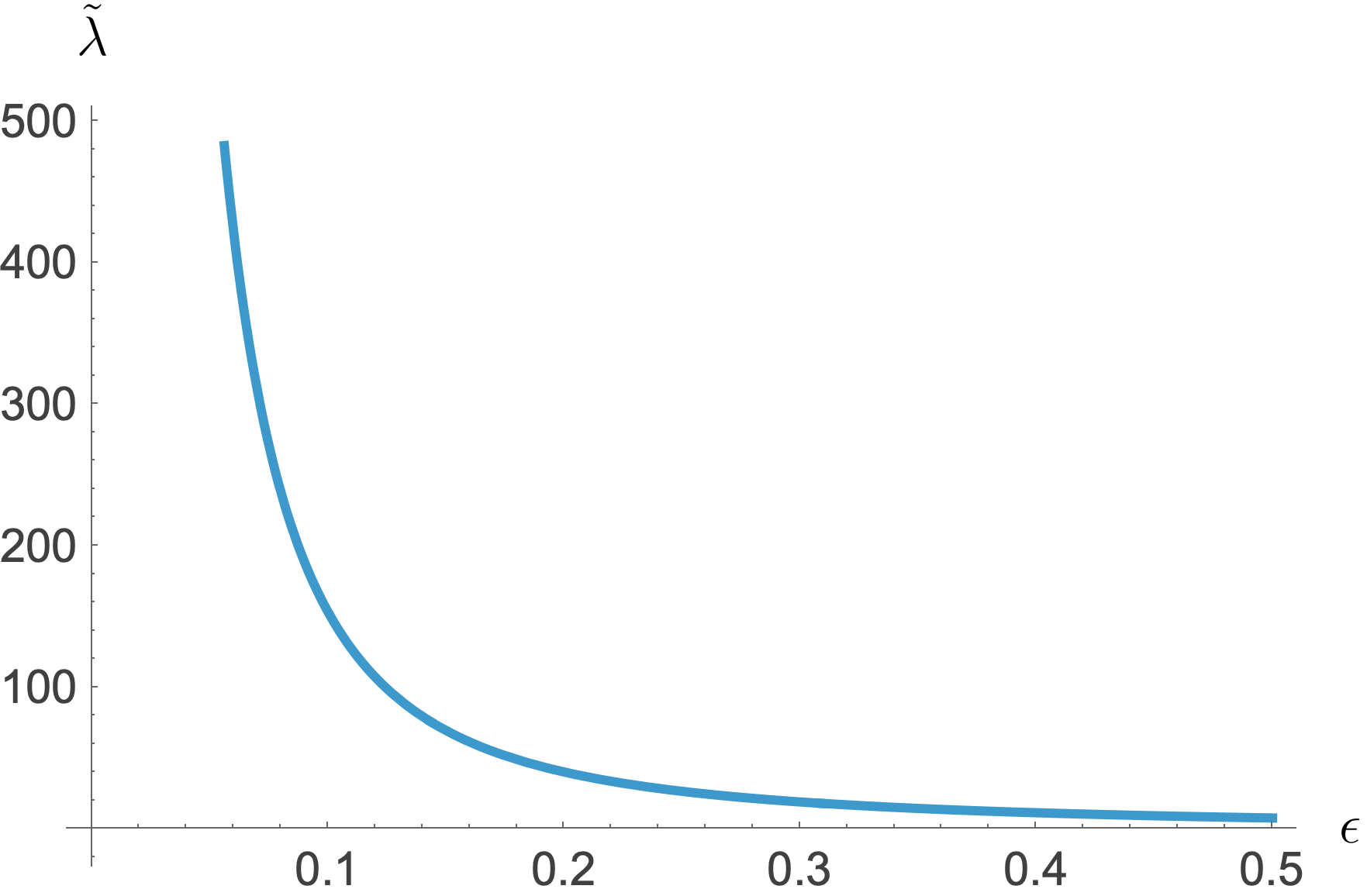}
        \caption{$(N_\alpha,N_\psi) = (1,1), (-1,-1)$.}
        \label{fig:sub2}
    \end{subfigure}
    \caption{Exceptional cases for $q=1$.}
    \label{fig3}
\end{figure}

To summarize, we have argued that the stability operator $L$ for all $(p,q)$ admits a single negative eigenvalue and hence the full class of homogeneous minimal hypersurfaces considered here has index 1. More precisely, we have derived an explicit formula for the eigenvalues of the stability operator. We then analytically proved that, outside of a small set of exceptional cases (for which our estimates on the parameters were insufficient to show positivity), there is only one negative eigenvalue. The absence of any further negative eigenvalues arising from the set of exceptional cases is established numerically. This demonstrates that the minimal hypersurface located at $y = \bar{y}$ has index 1, which establishes Theorem \ref{Thm2}.

\subsection*{Acknowledgements} The authors would like to thank Marcus Khuri for discussions and pointing out useful references and the referee for finding crucial errors in the original version and suggesting substantial improvements.


\begin{thebibliography}{5}

\bibitem{AdSCFT}
O.~Aharony, S.~S.~Gubser, J.~M.~Maldacena, H.~Ooguri and Y.~Oz, Large N field theories, string theory and gravity,
Phys. Rept. \textbf{323}, 183-386 (2000).


\bibitem{Alcubierre} M. Alcubierre, \emph{Introduction to 3+1 numerical relativity}, Oxford University Press, International series of monographs on physics (2008).

\bibitem{Besse} A.~L. Besse, {\it Einstein manifolds}, reprint of the 1987 edition, 
Classics in Mathematics, Springer, Berlin (2008). 

\bibitem{Bohm} C. B\"ohm and M.~M. Kerr, Low-dimensional homogeneous Einstein manifolds, Trans. Amer. Math. Soc. {\bf 358} ,  no.~4, 1455--1468 (2006).

\bibitem{Bohm2} C. B\"ohm, M.~K.~Y.-K. Wang and W. Ziller, A variational approach for compact homogeneous Einstein manifolds, Geom. Funct. Anal. {\bf 14}, no.~4, 681--733 (2004).

\bibitem{Chodosh} O. Chodosh, C. Li, and D. Stryker, Complete stable minimal hypersurfaces in positively curved 4-manifolds. J. Eur. Math. Soc. (2024).


\bibitem{survey} O. Chodosh and D. M\'aximo, The Morse index of a minimal surface, Notices Amer. Math. Soc. {\bf 68},  no.~6, 892--898 (2021).

\bibitem{Gauntlett:2004yd}
J.~P.~Gauntlett, D.~Martelli, J.~Sparks and D.~Waldram, Sasaki-Einstein metrics on $S^2 \times S^3$,
Adv. Theor. Math. Phys. \textbf{8}, no.4, 711-734 (2004).


\bibitem{Gibbons:2004em}
G.~W.~Gibbons, S.~A.~Hartnoll and Y.~Yasui,
Properties of some five dimensional Einstein metrics,
Class. Quant. Grav. \textbf{21}, 4697-4730 (2004).

\bibitem{Hennigar}
R.~A.~Hennigar, H.~K.~Kunduri, K.~T.~B.~Sievers and Y.~Wang, Spectrum of the Laplacian on the Page metric,
J. Phys. A \textbf{58}, no.40, 405204 (2025).


\bibitem{Hollands}
S.~Hollands and S.~Yazadjiev,
A Uniqueness theorem for stationary Kaluza-Klein black holes,
Commun. Math. Phys. \textbf{302}, 631-674 (2011).

\bibitem{KalafatSari} M. Kalafat and R. Sar\i, On special submanifolds of the Page space, Differential Geom. Appl. {\bf 74}, 13 pp (2021).

\bibitem{Kalafat:2024ctd}
M.~Kalafat, {\"O}.~Kelek{\c{c}}i and M.~Ta{\c{s}}demir,
Minimal submanifolds and stability in Einstein manifolds,
[arXiv:2406.03347 [math.DG]].

\bibitem{Kihara:2005nt}
H.~Kihara, M.~Sakaguchi and Y.~Yasui, Scalar Laplacian on Sasaki-Einstein manifolds $Y^{p,q}$,
Phys. Lett. B \textbf{621}, 288-294 (2005).


\bibitem{Kunduri:2006qa}
H.~K.~Kunduri, J.~Lucietti and H.~S.~Reall,
Gravitational perturbations of higher dimensional rotating black holes: Tensor perturbations,
Phys. Rev. D \textbf{74}, 084021 (2006).

\bibitem{Martelli:2004wu}
D.~Martelli and J.~Sparks,
Toric geometry, Sasaki-Einstein manifolds and a new infinite class of AdS/CFT duals,
Commun. Math. Phys. \textbf{262}, 51-89 (2006). 

\bibitem{Myers} S. B. Myers, Riemannian manifolds with positive mean curvature, Duke Math. J. 8(2), 401-404 (1941).

\bibitem{Page}
D.~N.~Page, A compact rotating gravitational instanton,
Phys. Lett. B \textbf{79}, 235-238 (1978).

\bibitem{Simons} J.~H. Simons, Minimal varieties in Riemannian manifolds, Ann. of Math. (2) {\bf 88},  62--105 (1968).

\bibitem{Sparks}
J.~Sparks, Sasaki-Einstein Manifolds,
Surveys Diff. Geom. \textbf{16}, 265-324 (2011).

\bibitem{Wu:1976ge}
T.~T.~Wu and C.~N.~Yang,
Dirac Monopole Without Strings: Monopole Harmonics,
Nucl. Phys. B \textbf{107}, 365 (1976).

\end{thebibliography}
\end{document}